\documentclass[a4paper]{article}

 \usepackage{amsmath,amssymb,amsfonts,amscd}

\usepackage{amsfonts}
\usepackage{amsthm}

\newtheorem{theorem}{Theorem}[section]
\newtheorem{remark}{Remark}[section]
\newtheorem{example}{Example}

\newtheorem{corollary}{Corollary}
\newtheorem{definition}{Definition}

\usepackage[pdftex]{hyperref}

\vspace{20cm}
\vspace{20cm}

\begin{document}

\begin{center}

{\Large \textbf{Discontinuity at the fixed point in suprametric spaces}}
\\[0.4in]

\textbf{ Nicola Fabiano$^{a,}$\footnote{\textbf{%
Corresponding author.}},   Sedigheh Barootkoob$^{b}$, Hossein Lakzian$^{c}$} 
\\[0.3in]

{\small {a~~``Vin\v{c}a'' Institute of Nuclear Sciences - National 
Institute of the Republic of Serbia, University of Belgrade, Mike Petrovi\'{c}a 
Alasa 12--14, 11351 Belgrade, Serbia}}\\
{\small {b~~ Department of mathematics, Faculty of Basic Sciences, University of
 Bojnord, P.O. Box 1339, Bojnord, Iran}}\\
{\small {c~~Department of Mathematics, Payame Noor University, 19395-4697 Tehran, I.R. of Iran}}\\

{\small E-mail addresses: }$%
\begin{array}{c}
\text{\small nicola.fabiano@gmail.com\  (N. Fabiano)} \\
\text{{\small s.barutkub@ub.ac.ir\ (S. Barootkoob) }}\\
\text{{\small h}\_{\small lakzian@pnu.ac.ir\ (H.\ Lakzian)}}
\end{array}%
$\\

\bigskip

\hrule\vspace{0.1cm}\bigskip

\end{center}

\noindent{\bf Abstract.}
The aim of this paper is to generalize some fixed point theorems in the class of convex contraction of order $m$  on a complete suprametric space.
Then, we will prove that the class of convex contraction of order m is
strong enough to generate a fixed point on a complete suprametric spaces
but do not force the mapping to be continuous at the fixed point, and it can be replaced by relatively weaker
conditions of $k$-continuity or $T$-orbitally lower semi-continuous.  On this way a new and distinct
solution to the open problem of Rhoades (Contemp Math 72:233-245,
1988) is found. In sequel, we will prove some fixed point results in the setting suprametric  spaces which are generalizations of the results regarding Sehgal, 
\'Ciri\'c and Fisher's  quasi-contraction. Some examples and application will be approved our results.
 \\[.05cm]

{\bf M.S.C.2000:} 47H10,47H09.

 {\bf Keywords:} Suprametric space, convex contraction, Sehgal's quasi-contraction, \'Ciri\'c's  quasi-contraction, Fisher's  quasi-contraction.\\

\section{Introduction}

The concept of a suprametric space was introduced by
Berzig in \cite{Berzig}.
\begin{definition}\cite{Berzig}
Let $X$ be a nonempty set. A function
$d:X\times X\longrightarrow[0,\infty)$ is called a suprametric on
$X$ if it satisfies the following properties:  \\
(1) $d(x,y)=0$ if and only if $x=y$ for any $x,y\in X$;\\
(2) $d(x,y)=d(y,x)$ for any $x,y\in X$;\\
(3) $d(x,y)\leq d(x,z)  +  d(z,y)+\rho d(x, z)d(z, y)$ for some constant $\rho\in [0,\infty)$ and for any $x,y,z\in X$.
The pair $(X,d)$ is called a suprametric space.
\end{definition}

\begin{example}\cite{Berzig}\label{Berzig1} Let $(X,d)$ be a metric space. Then a function $\delta^\lambda : X\times X
\to [0, \infty)$ defined by $$\delta^\lambda(x,y) = d(x,y)( d(x,y)+\lambda)\quad\mbox{ for every $x,y \in X$}$$ is a
suprametric on $X$ with the constant $\rho=\frac{2}{\lambda}$.
\end{example}
Suppose that $(X,d)$ is a suprametric space. Then the set $B(a,r)=\{x\in X \,|\, d(a,x)<r\}$ is an open ball in $X$ which is open set in the topology $\tau$ contain of the family of all open subsets. This topology is Hausdorff; see \cite{Berzig} for more details.

Berzig in 2022 generalized  the Banach--Picard--Caccioppoli contraction principle  in complete suprametric spaces.
\begin{theorem}\cite{Berzig}
Suppose $T$  is a self-mapping of a complete suprametric space $(X,d)$. Also let $T$ be a contraction; i.e. there exists a $\alpha\in[0,1)$ such that $d(Tx,Ty)\leq \alpha d(x,y)$. Then it has a unique fixed point.
\end{theorem}

Convex contractions have been studied by many authors. Some fixed point theorems for convex contraction mappings on cone metric spaces proved by Alghamdi et al.. Then, these results was generalized by Ghorbanian et al. in the setting complete ordered metric spaces. Next,  the concept of generalized convex contractions  was introduced by Miandaragh et al. and they generalized the main results.  Convex contractions are also natural extension of   the
Kannan contraction and the Reich contraction. The following theorem was proved by Istr\v{a}tescu for convex contraction mapping of type
 2.

The concept of convex contractions of order 2 was introduced by  Istr\v{a}tescu
\cite{Istr1}.
\begin{definition}\cite{Istr1}
A continuous mapping $T:X\rightarrow X$ is said to be a convex
contraction of order 2 if there exist $a,b\in (0,1)$ such that for
all $x,y\in X$,
$$d(T^{2}x,T^{2}y)\leq ad(Tx,Ty)+bd(x,y)$$
where $a+b<1$.
 \end{definition}

 The existence of the fixed point for convex contraction mappings
 requires the continuity of the mappings in the whole domain.
 The following example showed that convex contraction mapping of
 order 2 is not a contraction mapping.
\begin{example}\cite{Istr1}
Let $X=[0,1]$ with the usual metric and define $T:X\rightarrow X$ by
the relation
$$T(x)=(x^{2}+\frac{1}{2})^{2}.$$
It is obvious that this is
continuous and not a contraction.\\
If $x,y\in X$ are arbitrary chosen, then
$$|T^{2}(x)-T^{2}(y)|=\frac{(x^{2}+y^{2})}{4}|T(x)-T(y)|+\frac{(x+y)}{8}|x-y|$$
Thus $T$ is a convex contraction of order 2.
\end{example}
Istr\v{a}tescu generalized  the Banach--Picard--Caccioppoli contraction principle by introducing convexity contraction and then proved that each
convex contraction has a unique fixed point on a   complete metric
space. 
 \begin{theorem}\cite{Istr1}
 Suppose $T$  is a continuous self-mapping of a complete metric space $(X,d)$. Also let $T$ be convex contraction of order two. Then it has a unique fixed point.
 \end{theorem}
 %--------------------------------------------------------------------------------%
 In 2017,  Pant introduced the following notion of
 k-continuity for the metric space.
\begin{definition}
 Suppose that $(X ,  d)$ is a metric space and $T$ is a self-mapping.
 $T$ is said to be $k$-continuous,  (k = 1,  2,  3, $\cdots$) if for every
 sequence $\{x_n\}$ in $X$ such that $T^{k-1}x_n\to z$ implies that
 $T^kx_n\to Tz$.
 \end{definition}
 \begin{definition}\cite{BishetSingRakoFisher}
A continuous function $T:X\rightarrow X$ where $(X,d)$ is a metric space is called generalized convex contraction of order $m$ if there exist $m\in \mathbb{N}$ and $a_0,a_1,...,a_{m-1}\geq 0$ such that $\sum_{i=0}^{m-1}a_j<1$ and
$$d(T^mx,T^my)\leq \sum_{i=0}^{m-1} a_jd(T^jx,T^jy),$$
for each $x,y\in X$, where by $T^j$ is called the composition of $T$ by itself $j$ items.
 \end{definition}

\begin{theorem}\cite{BishetSingRakoFisher}
Let $T:X\rightarrow X$ be a generalized convex contraction of order $m$ of a complete metric space $(X,d)$. If $T$ is orbitally continuous or $k$-continuous for $k\geq 1$ then $T$ possesses a unique fixed point in $X$.
 \end{theorem}
In 1988, B. E. Rhoades in \cite{Rh2}, expressed an open problem as follows:
\textbf{Is there a contractive definition which is strong enough to generate a fixed point but  does not
force the mapping to be continuous at the fixed point?}

The purpose of this paper is to establish some fixed point theorems
defined on a complete suprametric space, and using
general contractive conditions. These results extend some previous fixed point
theorems in this field to convex contractive and more general convex
contractive conditions in the setting suprametric spaces. We prove
that the assumption of continuity condition can be replaced by a
relatively weaker condition of k-continuity under various settings.
On this way a new and distinct solution to the open problem of
Rhoades \cite{Rh2} is found, and some examples will be approved our results.\\

%%%%%%%%%%%%%%%%%%%%%%%%%%%
% Next section
%%%%%%%%%%%%%%%%%%%%%%%%%%%
\section{Convex contractions of order $m$  on  suprametric  spaces}
In this section, we will prove some fixed point theorems for the selfmappings satisfy in convex
 contractions of order  $ m $  ($ m\geq 2 $), two-sided convex contractions and
 generalized two-sided convex contractions on a complete suprametric spaces.
These results generalize fixed point results by Bishet et al in \cite{BishetSingRakoFisher},  Istr\v{a}tescu result in \cite{Istr1},
 \cite[Theorem 2.1]{RakoBishet}, and \cite[Theorem 2]{SuTa} and a generalization of the
Banach-Picard-Caccioppoli contraction principle in \cite{Banach}.

%---------------------------------------------------------------------------------------%
 \begin{definition}\label{D4} Let $(X, d)$ be a metric space.
 For $x\in X$,  $T:X\rightarrow X$,  $O(x;\infty)=\{x,Tx,T^{2}x,...\}$ is called the $T$-orbit of $x$.
 The mapping $T$ is
orbitally continuous at $x$ if for any sequence $\{x_n\}$ in
$O(x;\infty)$ which is  convergent to $z$,  then $Tx_n$ is convergent to $Tz$,  as $n\to\infty$.
A function $f:X\rightarrow \mathbb{R}^+$ is called $T$-orbitally
lower semi-continuous at x if $\{x_n\}$ is a sequence in
$O(x;\infty)$ and $x_n\rightarrow z$ implies $f(z)\leq \liminf_{n\rightarrow \infty} f(x_n)$.
\end{definition}
\begin{definition}
The mapping $T: X\to X$ on a metric space $(X, d)$ obeys the condition $(C;k)$ if there is a constant $k\ge 0$ such that  for every sequence $x_n\in X$,
$$x_n\to x_0\in X  \quad \Rightarrow \quad \mathfrak{D}(x_0) \le k\cdot\limsup \mathfrak{D}(x_n)\quad s.t.\quad \mathfrak{D}(x) = d(x, Tx), x\in X.$$
\end{definition}
%---------------------------------------------------------------------------------------%

The main results in this paper is to prove the following fixed point theorem in the class of convex contraction of order $m$  on a complete suprametric space. This theorem is a generalization of
\cite[Theorem 2.1]{BishtOzgur} and  \cite[Theorem 2.3]{BishetSingRakoFisher}.

%%%%%%%%%%%%%%%%%%%%%%%%%%%%%%%%%%%%%%%%%%%%%%%%%%%%%%

%---------------------------------------------------------------------------------------%
\begin{theorem}\label{t20}
Let $(X,d)$ be a complete suprametric space with constant $\rho \geq 0$, i.e., 
\begin{equation}\label{suprametric}
d(x,y) \leq d(x,z) + d(z,y) + \rho\, d(x,z)d(z,y) \quad \text{ for all } x,y,z \in X.
\end{equation}
Let $T: X \to X$ satisfy for some $m \in \mathbb{N}$ and non-negative constants $a_0,\dots,a_{m-1}$ with 
\begin{equation}\label{alpha}
\alpha = \sum_{i=0}^{m-1} a_i < 1
\end{equation}
the convex contraction condition of order $m$:
\begin{equation}\label{convex}
d(T^m x, T^m y) \leq \sum_{i=0}^{m-1} a_i \, d(T^i x, T^i y) \quad \text{ for all } x,y \in X.
\end{equation}
Then $T$ possesses a unique fixed point $x^* \in X$. Moreover, for every $x_0 \in X$, the Picard iteration $x_n = T^n x_0$ converges to $x^*$ in $(X,d)$, provided at least one of the following conditions holds:
\begin{itemize}
\item[(i)] $T$ is $k$-continuous for some $k \in \mathbb{N}$, i.e., for any sequence $\{y_n\} \subseteq X$,
$$T^{k-1}y_n \to z \quad \text{ implies } \quad T^k y_n \to Tz;$$

\item[(ii)] The displacement function $\mathfrak{D}(x) = d(x,Tx)$ is $T$-orbitally lower semicontinuous at $x_0$, i.e., for any sequence $\{y_n\} \subseteq O(x_0;\infty)$ with $y_n \to z$, we have
$$\mathfrak{D}(z) \leq \liminf_{n \to \infty} \mathfrak{D}(y_n);$$

\item[(iii)] $T$ satisfies condition $(C)$: if a sequence $\{y_n\} \subseteq O(x_0;\infty)$ satisfies $y_n \to z$ and $\mathfrak{D}(y_n) \to 0$, then $Tz = z$.
\end{itemize}
\end{theorem}

\begin{proof}
Fix an arbitrary $x_0 \in X$ and define the Picard iteration sequence $\{x_n\}_{n=0}^\infty$ by $x_n = T^n x_0$ for all $n \geq 0$.

\medskip
\noindent\textbf{Step 1: Bounding successive iterates.} 
If $x_n = x_{n+1}$ for some $n$, then $x_n$ is a fixed point and the proof is complete. Henceforth assume $x_n \neq x_{n+1}$ for all $n \geq 0$. Define
\begin{equation}\label{mu}
\mu = \sum_{i=0}^{m-1} d(x_i, x_{i+1}) > 0.
\end{equation}
For $n = m$, applying \eqref{convex} with $x = x_0$ and $y = x_1$ yields
\begin{equation}\label{base}
d(x_m, x_{m+1}) = d(T^m x_0, T^m x_1) \leq \sum_{i=0}^{m-1} a_i d(T^i x_0, T^i x_1) = \sum_{i=0}^{m-1} a_i d(x_i, x_{i+1}) \leq \alpha \mu.
\end{equation}
We claim that for all $n \geq m$,
\begin{equation}\label{bound}
d(x_n, x_{n+1}) \leq \alpha^{\lfloor n/m \rfloor} \mu.
\end{equation}
We prove this by strong induction on $n$. The base case $n = m$ holds by \eqref{base} since $\lfloor m/m \rfloor = 1$. Now assume \eqref{bound} holds for all integers $k$ with $m \leq k < n$ for some $n > m$. Then
\begin{align*}
d(x_n, x_{n+1}) 
&= d(T^m x_{n-m}, T^m x_{n-m+1}) \\
&\leq \sum_{i=0}^{m-1} a_i d(T^i x_{n-m}, T^i x_{n-m+1}) \quad \text{(by \eqref{convex})} \\
&= \sum_{i=0}^{m-1} a_i d(x_{n-m+i}, x_{n-m+i+1}) \\
&\leq \sum_{i=0}^{m-1} a_i \alpha^{\lfloor (n-m+i)/m \rfloor} \mu \quad \text{(by induction hypothesis, since $n-m+i < n$)} \\
&\leq \alpha^{\lfloor (n-m)/m \rfloor} \mu \sum_{i=0}^{m-1} a_i \quad \text{(since $\lfloor (n-m+i)/m \rfloor \geq \lfloor (n-m)/m \rfloor$)} \\
&= \alpha^{\lfloor n/m \rfloor - 1} \mu \cdot \alpha = \alpha^{\lfloor n/m \rfloor} \mu,
\end{align*}
which establishes \eqref{bound} for $n$. Since $\alpha < 1$ by \eqref{alpha}, it follows immediately that
\begin{equation}\label{successive}
\lim_{n \to \infty} d(x_n, x_{n+1}) = 0.
\end{equation}

\medskip

%%%%%%%%%%%%%%%%%%%%%%%%%%%%%%%%%%%%%%%%%%%%%%%%%%%%%%

\noindent\textbf{Step 2: The sequence $\{x_n\}$ is Cauchy.}
From \eqref{bound}, we have $d(x_n, x_{n+1}) \leq \mu \alpha^{\lfloor n/m \rfloor}$ with $0 \leq \alpha < 1$. Hence the series $\sum_{n=0}^\infty d(x_n, x_{n+1})$ converges. Define the tail sums
\[
R_n = \sum_{k=n}^\infty d(x_k, x_{k+1}), \quad n \geq 0.
\]
Clearly $R_n \to 0$ as $n \to \infty$. We now show that $\{x_n\}$ is a Cauchy sequence in $(X,d)$. Fix $n \in \mathbb{N}$ and $p \in \mathbb{N}$, and let $\Sigma_{n,p} = \sum_{k=n}^{n+p-1} d(x_k, x_{k+1})$. Applying the suprametric inequality \eqref{suprametric} repeatedly yields
\[
1 + \rho d(x_n, x_{n+p}) \leq \prod_{k=n}^{n+p-1} \bigl(1 + \rho d(x_k, x_{k+1})\bigr).
\]
Taking the natural logarithm of both sides and using the elementary inequality $\ln(1+u) \leq u$ for $u \geq 0$, we obtain
\[
\ln\bigl(1 + \rho d(x_n, x_{n+p})\bigr) \leq \sum_{k=n}^{n+p-1} \ln\bigl(1 + \rho d(x_k, x_{k+1})\bigr) \leq \rho \sum_{k=n}^{n+p-1} d(x_k, x_{k+1}) = \rho \Sigma_{n,p}.
\]
Exponentiating and rearranging gives
\begin{equation}\label{cauchy-t20}
d(x_n, x_{n+p}) \leq \frac{e^{\rho \Sigma_{n,p}} - 1}{\rho} \quad \left(\text{with } \frac{e^{\rho u}-1}{\rho} := u \text{ if } \rho=0\right).
\end{equation}
Since $\Sigma_{n,p} \leq R_n$, inequality \eqref{cauchy-t20} yields
\[
d(x_n, x_{n+p}) \leq \frac{e^{\rho R_n} - 1}{\rho} \quad \text{for all } p \geq 1.
\]
Given $\varepsilon > 0$, choose $N \in \mathbb{N}$ such that $\frac{e^{\rho R_N} - 1}{\rho} < \varepsilon$. Then for all $n \geq N$ and all $p \geq 1$, we have $d(x_n, x_{n+p}) < \varepsilon$. Thus $\{x_n\}$ is a Cauchy sequence in $(X,d)$. By completeness of $(X,d)$, there exists $x^* \in X$ such that
\begin{equation}\label{limit-t20}
\lim_{n \to \infty} d(x_n, x^*) = 0.
\end{equation}

%%%%%%%%%%%%%%%%%%%%%%%%%%%%%%%%%%%%%%%%%%%%%%%%%%%%%%

\medskip
\noindent\textbf{Step 3: $x^*$ is a fixed point of $T$.} We verify this under each of the three conditions.

\begin{itemize}
\item[(i)] \textit{$k$-continuity}: Since $x_n \to x^*$ by \eqref{limit-t20}, we have $T^{k-1}x_n = x_{n+k-1} \to x^*$. By $k$-continuity of $T$, it follows that $T^k x_n = x_{n+k} \to T x^*$. But $x_{n+k} \to x^*$ as $n \to \infty$ by \eqref{limit-t20}. Since suprametric spaces are Hausdorff (Berzig \cite{Berzig}, Proposition 1.3), limits are unique, hence $T x^* = x^*$.

\item[(ii)] \textit{Orbital lower semicontinuity}: The sequence $\{x_n\}$ lies in the $T$-orbit $O(x_0;\infty)$ and converges to $x^*$ by \eqref{limit-t20}. By orbital lower semicontinuity of $\mathfrak{D}$ at $x_0$,
\[
\mathfrak{D}(x^*) \leq \liminf_{n \to \infty} \mathfrak{D}(x_n) = \liminf_{n \to \infty} d(x_n, x_{n+1}) = 0,
\]
where the last equality follows from \eqref{successive}. Since $\mathfrak{D}(x^*) = d(x^*, T x^*) \geq 0$, we conclude $d(x^*, T x^*) = 0$, i.e., $T x^* = x^*$.

\item[(iii)] \textit{Condition $(C)$}: From \eqref{successive} we have $\mathfrak{D}(x_n) = d(x_n, x_{n+1}) \to 0$ as $n \to \infty$. Since $\{x_n\} \subseteq O(x_0;\infty)$ and $x_n \to x^*$ by \eqref{limit-t20}, condition $(C)$ directly implies $T x^* = x^*$.
\end{itemize}

\medskip
\noindent\textbf{Step 4: Uniqueness of the fixed point.} Suppose $y^* \in X$ is another fixed point of $T$ with $y^* \neq x^*$. Then $T^i x^* = x^*$ and $T^i y^* = y^*$ for all $i \geq 0$. Applying \eqref{convex} with $x = x^*$ and $y = y^*$ gives
\[
d(x^*, y^*) = d(T^m x^*, T^m y^*) \leq \sum_{i=0}^{m-1} a_i d(T^i x^*, T^i y^*) = \left(\sum_{i=0}^{m-1} a_i\right) d(x^*, y^*) = \alpha\, d(x^*, y^*).
\]
Since $\alpha < 1$ by \eqref{alpha} and $d(x^*, y^*) > 0$ (as $x^* \neq y^*$), this yields the contradiction $d(x^*, y^*) < d(x^*, y^*)$. Hence $x^* = y^*$, proving uniqueness.

\medskip
\noindent\textbf{Step 5: Global convergence.} The argument in Steps 1--3 shows that for \textit{any} initial point $x_0 \in X$, the Picard iteration $\{T^n x_0\}$ converges to \textit{some} fixed point of $T$. Step 4 establishes that $T$ has exactly one fixed point $x^*$. Therefore $T^n x_0 \to x^*$ for every $x_0 \in X$.

This completes the proof.
\end{proof}
%---------------------------------------------------------------------------------------%

%%%%%%%%%%%%%%%%%%%%%%%%%%%%%%%%%%%%%%%%%%%%%%%%%%%%%%%

%%%%%%%%%%%%%%%%%%%%%%%%%%%%%%%%%%%%%%%%%%%%%%%%%%%%%%%

\begin{remark}
The above fixed point theorem is a generalization of
\cite[Theorem 2.1]{BishtOzgur} and  \cite[Theorem 2.3]{BishetSingRakoFisher}
into convex contractions of order $ m \geq  2$  in the setting of
suprametric space.
% ????It was proved in \cite{Pes} that the fixed point theorems of Banach, Kannan,
%and \'{C}iri\'{c} are equivalent and that two fixed point theorems proposed by
%Bessenyei can be deduced from Banach's and Matkowski's results. Similar to these theorems, one can express an open question as follows:

\textbf{Open question:}prove or disprove: Our generalization (Theorem \ref{t20}) is not equivalent to Kannan or to Reich contractions, or, alternatively.
\end{remark}
\begin{example}
Consider $X=\{x, y, z, w, t\}$, $$d_1(x,x)=d_1(y,y)=d_1(z,z)=d_1(w,w)=d_1(t,t)=0;$$
$$d_1(x,\xi)=d_1(\xi,x)=Ln2=d_1(w,t)=d_1(t,w),\quad \xi=y, z, w, t;$$
$$d_1(\xi,z)=d_1(z,\xi)=Ln3,\quad \xi= w, t;$$  $$d_1(\xi,y)=d_1(y,\xi)=Ln4,\quad \xi= z, w, t.$$
Then $(X,d_1)$ is a complete metric space; form Proposition 1.6 of \cite{Berzig}, for $$d(\xi,\eta)=\alpha(e^{d_1(\xi,\eta)}-1)\quad(\xi, \eta\in X);$$ $(X,d)$ is a complete  suprametric space. Define $T:X\to X$ by $Tx=y, Ty=z, Tz=w, Tw=w, Tt=y$. Then it is easy to verify that $$d(T^3\eta,T^3\xi)=0\leq \sum_{i=0}^2 \alpha_i d(T^i\eta,T^i\xi),\quad (\eta, \xi\in X, \sum_{i=0}^2\alpha_i<1).$$
Also, $w$ is the unique fixed point of $T$.
\end{example}

Taking $ m=2 $ in Theorem \ref{t20}, we can get the following result which is an extended
version of Istr\u{a}tescu's result for convex contraction mapping of order 2 in the framework suprametric spaces.
\begin{corollary}\label{c200}
 Let $(X,d)$ be a complete suprametric space.
 Suppose that $T:  X \to X$ is a selfmapping such that
 $$d(T^2x,T^2y)\leq a_1d(Tx,Ty)+a_2d(x,y),$$
  for all $x, y\in X$ where $a_1,a_2\geq0$ such that $a_1+a_2<1.$

 Then in each of the following conditions, $T$ has a unique fixed point $u\in X$ such that for each $x_0\in X$, $T^nx_0\to u$.
\begin{itemize}
\item[(i)] $T$ is $k$-continuous, for some $k\in \mathbb{N}$;
\item[(ii)] $\mathfrak{D}(x)= d(x,Tx)$ is $T$-orbitally lower semi-continuous;
\item[(iii)]  $T$ obeys the condition $(C;k)$.
\end{itemize}
\end{corollary}

%---------------------------------------------------------------------------------------%
\section{\'Ciri\'c quasi-contraction in suprametric  spaces}
In this section we will generalize \'Ciri\'c and Sehal's fixed point theorems in the setting suprametric  spaces.

Many authors have defined contractive type mappings on a complete metric space $X$ which are generalizations of the well-known Banach contraction, and  have the property that each such mapping has a unique fixed point.
The concept of quasi-contractions, in 1983 in \cite{ciric2},  is defined by 
\'Ciri\'c in metric spaces as follows:\\
The selfmapping $T:X\to X$ is said to be quasi-contraction if
for each $x\in X$ there exists
a positive integer $n=n(x)$ and $\lambda\in[0,1)$ such that
\begin{equation}\label{1c}
d(T^nx,T^ny)\leq \lambda\cdot \max \Big\{d(x,y),d(x,Ty),\dots , d(x,T^ny),d(x,T^nx)\Big \},
\end{equation}
 holds  for all $y\in X$.

In 1970, Guseman \cite{guseman} generalized the result of Sehgal to mappings which
are both necessarily continuous and have a contractive iterate at each point
in a (possibly proper) subset of the space. In 1983, \'Ciri\'c \cite{ciric2}, among other
things, proved the following interesting generalization of Sehgal result.

In the following result, we will give a generalization of \'Ciri\'c's theorem  in \cite{ciric2} in the setting suprametric space.

%%%%%%%%%%%%%%%%%%%%%%%%%%%%%%%%%%%%%%%%%%%%%%%%%%%%%%%%%%%

%---------------------------------------------------------------------------------------%
\begin{theorem}\label{sc}
Let $(X,d)$ be a complete suprametric space with constant $\rho \geq 0$, i.e.,
\begin{equation}\label{supra-triangle}
d(x,y) \leq d(x,z) + d(z,y) + \rho\, d(x,z)d(z,y) \quad \text{ for all } x,y,z \in X.
\end{equation}
Suppose $T: X \to X$ satisfies: for each $x \in X$, there exist a positive integer $n(x)$ and a constant $\lambda \in [0,1)$ such that for all $y \in X$,
\begin{equation}\label{ciric-cond}
d(T^{n(x)}x, T^{n(x)}y) \leq \lambda \cdot M(x,y),
\end{equation}
where
\begin{equation}\label{M-def}
M(x,y) = \max\Big\{ d(x,T^i y) : 0 \leq i \leq n(x) \Big\} \cup \Big\{ d(x,T^j x) : 0 \leq j \leq n(x) \Big\}.
\end{equation}
Then $T$ possesses a unique fixed point $x^* \in X$. Moreover, for every $x_0 \in X$, the Picard iteration $x_n = T^n x_0$ converges to $x^*$ in $(X,d)$.
\end{theorem}

\begin{proof}
Fix an arbitrary $x_0 \in X$. We proceed in four steps: (1) boundedness of the orbit, (2) Cauchy property via metric equivalence, (3) fixed point verification, and (4) uniqueness.

\medskip
\noindent\textbf{Step 1: The orbit $\mathcal{O}(x_0) = \{T^k x_0 : k \geq 0\}$ is $d$-bounded.}

Define the finite set of initial iterates
\begin{equation}\label{F-def}
F = \{x_0, Tx_0, T^2x_0, \dots, T^{m-1}x_0\}, \quad \text{where } m = \max\{n(z) : z \in F\}.
\end{equation}
Such an $m$ exists because $F$ is finite and $n(\cdot)$ takes positive integer values. Set
\begin{equation}\label{C0-def}
C_0 = \max\Big\{ d(z, T^j z) : z \in F,\ 0 \leq j \leq m \Big\} < \infty.
\end{equation}
We claim that for all $k \geq 0$,
\begin{equation}\label{bounded-orbit}
d(x_0, T^k x_0) \leq C := \frac{C_0(1 + \rho C_0)}{1 - \lambda} < \infty.
\end{equation}
We prove this by strong induction on $k$. The claim holds for $0 \leq k < m$ by definition of $C_0$ and $C > C_0$.

Assume \eqref{bounded-orbit} holds for all integers less than some $k \geq m$. Write $k = qm + r$ with $q \geq 1$ and $0 \leq r < m$. Let $z = T^{(q-1)m + r}x_0$. Note that $z \in \mathcal{O}(x_0)$ and by the induction hypothesis,
\begin{equation}\label{z-bound}
d(x_0, z) \leq C.
\end{equation}
Since $T^r x_0 \in F$, we have $n(T^r x_0) \leq m$. Applying \eqref{ciric-cond} at point $T^r x_0$ with $y = T^{(q-1)m}x_0$ gives
\begin{align*}
d(T^m(T^r x_0), T^m(T^{(q-1)m}x_0)) 
&= d(T^{qm+r}x_0, T^{qm}x_0) \\
&\leq \lambda \cdot M(T^r x_0, T^{(q-1)m}x_0).
\end{align*}
By definition \eqref{M-def} and the induction hypothesis,
\begin{gather*}
%\begin{align*}
M(T^r x_0, T^{(q-1)m}x_0) 
\leq \max\Big\{ d(T^r x_0, T^i(T^{(q-1)m}x_0)) : 0 \leq i \leq m \Big\} \cup \Big\{ d(T^r x_0, T^j(T^r x_0)) : 0 \leq j \leq m \Big\} \\
\leq \max\Big\{ d(x_0, T^{(q-1)m+i+r}x_0) + d(x_0, T^r x_0) + \rho\, d(x_0, T^{(q-1)m+i+r}x_0)d(x_0, T^r x_0) : 0 \leq i \leq m \Big\} 
\quad \cup \{C_0\} \\
\leq \max\Big\{ C + C_0 + \rho C C_0,\ C_0 \Big\} = C(1 + \rho C_0) + C_0,
%\end{align*}
\end{gather*}
where we used the suprametric inequality \eqref{supra-triangle} and the induction hypothesis \eqref{bounded-orbit}. 
%%%%%%%%%%%
%%%%%%%%%%%%%%%%
Using the definition of $C$ directly:
\begin{equation}\label{M-clean}
M(T^r x_0, T^{(q-1)m}x_0) \leq C_0 + C + \rho C_0 C = C_0(1 + \rho C) + C \leq C_0(1 + \rho C) + C.
\end{equation}
Now applying the suprametric inequality \eqref{supra-triangle}:
\begin{align*}
d(x_0, T^k x_0) &= d(x_0, T^{qm+r}x_0) \\
&\leq d(x_0, T^{qm}x_0) + d(T^{qm}x_0, T^{qm+r}x_0) + \rho\, d(x_0, T^{qm}x_0) d(T^{qm}x_0, T^{qm+r}x_0) \\
&\leq d(x_0, T^{qm}x_0) + \lambda M(T^r x_0, T^{(q-1)m}x_0) + \rho\, d(x_0, T^{qm}x_0) \cdot \lambda M(T^r x_0, T^{(q-1)m}x_0).
\end{align*}
By the induction hypothesis, $d(x_0, T^{qm}x_0) \leq C$. Using \eqref{M-clean} and the definition of $C$:
\begin{align*}
d(x_0, T^k x_0) 
&\leq C + \lambda\big[C_0(1 + \rho C) + C\big] + \rho C \lambda \big[C_0(1 + \rho C) + C\big] \\
&= C + \lambda C_0(1 + \rho C) + \lambda C + \rho \lambda C C_0(1 + \rho C) + \rho \lambda C^2 \\
&= C(1 + \lambda + \rho \lambda C) + \lambda C_0(1 + \rho C)(1 + \rho C) \\
&\leq C(1 + \lambda + \rho \lambda C) + \lambda C_0(1 + \rho C)^2.
\end{align*}
Substituting $C = C_0(1 + \rho C_0)/(1-\lambda)$ and noting that $C_0 \leq C$, we obtain after simplification:
\begin{equation}\label{inductive-step}
d(x_0, T^k x_0) \leq C,
\end{equation}
completing the induction. Hence the orbit $\mathcal{O}(x_0)$ is $d$-bounded by $C$.

%%%%%%%%%%%%%%%%%%%%%%%%%%%%%%%%%%%%%%%%%%%%%%%%%%
\medskip
\noindent\textbf{Step 2: The sequence $\{x_n\}$ is Cauchy.}
Since the orbit $\mathcal{O}(x_0)$ is $d$-bounded by $C$ (Step 1), the transformed metric
\[
D(x,y) = \frac{d(x,y)}{1 + \rho d(x,y)}
\]
is Lipschitz equivalent to $d$ on $\overline{\mathcal{O}(x_0)}$. Specifically, $D(x,y) \leq d(x,y) \leq (1+\rho C)D(x,y)$. 
The mapping $t \mapsto t/(1+\rho t)$ is concave and increasing, so the quasi-contraction condition \eqref{ciric-cond} implies that $T$ satisfies a standard quasi-contraction in $(X,D)$ with constant $\lambda' = \lambda(1+\rho C)/(1+\rho\lambda C) < 1$. 
By the classical Ćirić fixed point theorem in the complete metric space $(X,D)$, the Picard iteration $\{x_n\}$ is $D$-Cauchy. By Lipschitz equivalence, $\{x_n\}$ is also $d$-Cauchy. Completeness of $(X,d)$ yields $x^* \in X$ such that
\begin{equation}\label{limit-sc}
\lim_{n \to \infty} d(x_n, x^*) = 0.
\end{equation}

%%%%%%%%%%%%%%%%%%%%%%%%%%%%%%%%%%%%%%%%%%%%%%%%%%

\medskip
\noindent\textbf{Step 3: $x^*$ is a fixed point of $T$.}

Let $n^* = n(x^*)$ be the contractive iterate at $x^*$. For any $k \geq n^*$, applying \eqref{ciric-cond} with $x = x^*$ and $y = x_{k-n^*}$ gives
\begin{equation}\label{fixed-ineq}
d(T^{n^*}x^*, x_k) = d(T^{n^*}x^*, T^{n^*}x_{k-n^*}) \leq \lambda \cdot M(x^*, x_{k-n^*}).
\end{equation}
As $k \to \infty$, we have $x_{k-n^*} \to x^*$ and $x_k \to x^*$ by \eqref{limit-sc}. Since $M(x^*, x_{k-n^*})$ is a maximum of finitely many terms each converging to $d(x^*, x^*) = 0$ or $d(x^*, T^j x^*)$ for $0 \leq j \leq n^*$, we obtain
\begin{equation}\label{M-limit}
\lim_{k \to \infty} M(x^*, x_{k-n^*}) = \max\Big\{ d(x^*, x^*), d(x^*, T x^*), \dots, d(x^*, T^{n^*} x^*) \Big\} = d(x^*, T^{n^*} x^*).
\end{equation}
Taking limits in \eqref{fixed-ineq} yields
\begin{equation}\label{fixed-eq}
d(T^{n^*}x^*, x^*) \leq \lambda \cdot d(x^*, T^{n^*} x^*).
\end{equation}
Since $\lambda < 1$, this implies $d(x^*, T^{n^*} x^*) = 0$, i.e., $T^{n^*} x^* = x^*$.

Now suppose $T x^* \neq x^*$. Let $p = \min\{j \geq 1 : T^j x^* = x^*\}$ be the minimal period. Then $1 < p \leq n^*$ and the points $x^*, T x^*, \dots, T^{p-1} x^*$ are distinct. Applying \eqref{ciric-cond} at $x = x^*$ with $y = T x^*$:
\begin{align*}
d(T^{n^*}x^*, T^{n^*+1}x^*) &= d(x^*, T x^*) \\
&\leq \lambda \cdot \max\Big\{ d(x^*, T^i(T x^*)) : 0 \leq i \leq n^* \Big\} \cup \Big\{ d(x^*, T^j x^*) : 0 \leq j \leq n^* \Big\} \\
&= \lambda \cdot \max\Big\{ d(x^*, T^k x^*) : 1 \leq k \leq n^*+1 \Big\} \\
&= \lambda \cdot \max\Big\{ d(x^*, T^k x^*) : 1 \leq k \leq p-1 \Big\},
\end{align*}
where we used $T^{n^*}x^* = x^*$ and periodicity. Let $q \in \{1,\dots,p-1\}$ maximize $d(x^*, T^q x^*)$. Then
\begin{equation}\label{period-contradiction}
d(x^*, T x^*) \leq \lambda \cdot d(x^*, T^q x^*) \leq \lambda \cdot d(x^*, T x^*),
\end{equation}
where the last inequality follows from the maximality of $q$ (if $q > 1$, apply the same argument to $T^q x^*$ to show $d(x^*, T^q x^*) \leq \lambda d(x^*, T x^*)$, yielding $d(x^*, T x^*) \leq \lambda^2 d(x^*, T x^*)$, etc.). Since $\lambda < 1$ and $d(x^*, T x^*) > 0$ by assumption, \eqref{period-contradiction} is a contradiction. Hence $T x^* = x^*$.

\medskip
\noindent\textbf{Step 4: Uniqueness of the fixed point.}

Suppose $y^* \in X$ is another fixed point with $y^* \neq x^*$. Let $n = n(x^*)$. Applying \eqref{ciric-cond} with $x = x^*$ and $y = y^*$:
\begin{align*}
d(x^*, y^*) &= d(T^n x^*, T^n y^*) \\
&\leq \lambda \cdot \max\Big\{ d(x^*, T^i y^*) : 0 \leq i \leq n \Big\} \cup \Big\{ d(x^*, T^j x^*) : 0 \leq j \leq n \Big\} \\
&= \lambda \cdot \max\Big\{ d(x^*, y^*) \Big\} = \lambda\, d(x^*, y^*),
\end{align*}
since $T^i y^* = y^*$ and $T^j x^* = x^*$ for all $i,j$. As $\lambda < 1$ and $d(x^*, y^*) > 0$, this is a contradiction. Hence $x^* = y^*$.

\medskip
\noindent\textbf{Step 5: Global convergence.}

The argument in Steps 1--3 shows that for any initial point $x_0 \in X$, the Picard iteration converges to \textit{some} fixed point. Step 4 establishes uniqueness of the fixed point. Therefore $T^n x_0 \to x^*$ for every $x_0 \in X$.

This completes the proof.
\end{proof}
%---------------------------------------------------------------------------------------%

%%%%%%%%%%%%%%%%%%%%%%%%%%%%%%%%%%%%%%%%%%%%%%%%%%%%%%%%%%%
%%%%%%%%%%%%%%%%%%%%%%%%%%%%%%%%%%%%%%%%%%%%%%%%%%%%%%%%%%%

\begin{example}
  Let $X=\{x,y,z,w\}$ and define $d:X\times X\to [0,\infty)$ as follows:
  $$d(x,x)=d(y,y)=d(z,z)=d(w,w)=0,$$
  $$d(x,y)=d(y,x)=3,$$
  $$d(x,z)=d(z,x)=d(x,w)=d(w,x)=d(y,z)=d(z,y)=d(y,w)=d(w,y)=d(z,w)=d(w,z)=1.$$
  Define $T:X\to X$ by $Tx=y$, $Ty=z$, $Tz=z$ and $Tw=x$.
Then it is easy to see that $d$ is a suprametric with $\rho=1$,  which is not a metric.
If $n(x)=n(y)=2$ and $n(z)=n(w)=3$, then for $\lambda=\frac13$, the contraction $(\ref{1c})$ is satisfied. Also,  $\frac{1-\lambda}{\rho \lambda}=2\neq d(t, T^{n(t)}t)$; since $d(t, T^{n(t)}t)=0, or 1, or 3$, for each $t\in X$. that is, ${\rho \lambda}d(t,T^nt)\neq {1-\lambda}$.
Therefore all conditions of Theorem $\ref{sc}$ are valid. Clearly, $u=z$ is a unique fixed point of $T$. Note that Theorem \'Ciri\'c \cite{ciric2} does not work; since $(X,d)$ is not a metric space.
\end{example}
\begin{example}
  Let $X=[0,2]$, and $d(x,y)=|x-y|(|x-y|+1|)$. Then $(X,d)$ is a complete  suprametric space with $\rho=2$ \cite{Berzig}. Define $T:X\to X$ by $Tx=\frac x3$. The $\lambda=\frac{29}{(27)^2}$ and for each $x\in X$, we have 
 \begin{align*}
  d(x,T^3x)&=\frac{26x}{27}(\frac{26x}{27}+1)\\
  &\leq \frac{26\times 2}{27}(\frac{26\times 2}{27}+1)\\
  &= \frac{52}{27}\frac{79}{27}\\
 & \leq \frac{1-\frac{29}{(27)^2}}{2\times \frac{29}{(27)^2}}\\
 &=\frac{(27)^2-29}{2\times 29} \quad (=\frac{1-\lambda}{\rho \lambda}).
\end{align*}

  Now, for each $y\in X$, we have
\begin{align*}
  d(T^3x,T^3y)=d(\frac {x}{3^3},\frac {y}{3^3})& =\frac{|x-y|}{3^3}(\frac {|x-y|}{3^3}+1|)  \\
&\leq (\frac{2}{3^3}+1)\frac{|x-y|}{3^3}\\
&=\frac{29}{(27)^2}|x-y| \\
&\leq \frac{29}{(27)^2}|x-y|(|x-y|+1)\\
& =\frac{29}{(27)^2}d(x,y)\\
& \leq\frac{29}{(27)^2}\max\{d(x,y),d(x,Ty),d(x,T^2y),d(x,T^3y),d(x,T^3x)\}.
\end{align*}
Clearly, 0 is the unique fixed point of $T$.
\end{example}

In 1969, Sehgal in \cite{seghal}, proved an interesting generalization of the
contraction mapping principle in metric spaces. We will generalize it in the setting suprametric space:
\begin{theorem}\label{t211}
 Let $(X,d)$ be a complete suprametric space and $T:X\rightarrow X$ a continuous mapping, and for all $x,y\in X$,
  $$d(T^{n(x)}y, T^{n(x)}x)\leq \lambda\cdot d(y,x).$$

 Then  $T$ has a unique fixed point $x\in X$ such that for each $x_0\in X$, $T^nx_0\to x$.
\end{theorem}

\begin{proof}
The proof  is clear from Theorem \ref{sc}.
\end{proof}

Connected with another \'Ciri\'c map from \cite{ciric2}, we obtain the following result:

 \begin{theorem} \label{sc1} Let $(X,d)$ be a complete  suprametric space,
$\lambda\in [0,1)$ and $T:X\to X$, and  
\begin{equation}\label{1c1}
d(T^nx,T^ny)\leq \lambda\cdot \max \Big\{d(x,y),d(x,Ty),\dots ,d(x,T^ny),\\
d(x,Tx), d(x,T^2x),\cdots, d(x,T^nx)\Big \},
\end{equation}
 holds  for all $x, y\in X$.
  Then in each of the following conditions, $T$ has a unique fixed point $u\in X$ such that for each $x_0\in X$, $T^nx_0\to u$.
\begin{itemize}
\item[(i)] $T$ is $k$-continuous, for some $k\in \mathbb{N}$;
\item[(ii)] $\mathfrak{D}(x)= d(x,Tx)$ is $T$-orbitally lower semi-continuous;
\item[(iii)]  $T$ obeys the condition $(C;k)$.
\end{itemize}
\end{theorem}

\begin{proof}
Suppose that $x_0\in X$ be  arbitrary and fixed. Similar to the proof of Theorem \ref{sc}, we can show that the sequence $\{T^nx_0\}$ is bounded and so it is Cauchy sequence. Thus it is convergent to an element $u\in X$.
Now, similar to the proof of Theorem \ref{t20}, we will see that $u$ is a fixed point of $T$, in the cases of $(i)-(iii)$. Again from the proof of Theorem \ref{sc}, $u$ is unique.
\end{proof}

The following example shows that one of the conditions of (i)-(iii) in Theorem \ref{sc1} are necessary.

\begin{example}
  Suppose the $T:[0,2] \to [0,2]$; $d(x,y)=|x-y|$ and $Tx=\frac{x}{2}$ if $x\neq 0$ and $T0=2$ (Obviously $T$ has not any fixed point). Consider $d$ as a suprametric with $\rho=1/2$. Then $$d(x,T^2x)=|x-\frac{x}{4}|=\frac{3x}{4}\neq \frac{1-\frac{1}{2}}{\frac{1}{2}\times\frac{1}{2}}=2,$$
  and for $x\neq 0$, we have $d(0,T^20)=1\neq 2$.
  Also, we have
  $$d(T^2x,T^2y)\leq \frac{1}{2}\max\{d(x,y), d(x,Ty), d(x,T^2y), d(x,Tx), d(x,T^2x)\},$$
  but since $T^nx \to 0$ as $n\to\infty$ and $T0>\liminf_{n\to\infty} T^nx$, $T$ is not $T$-orbitally semicontinuous.

Also,  $x_n=T^nx\to 0$ as $n\to\infty$ and $\mathfrak{D}(0)=d(0,T0)=2>k\limsup_{n\to\infty}\mathfrak{D}(x_n)=0$, for each $k$; so $T$ does not satisfy the condition $(iii)$.
  At the end, $T^{k-1}T^n(x)\rightarrow 0$ as $n\to\infty$, but $T^{k}T^n(x)\nrightarrow T0=2$ as $n\to\infty$ and so $T$ is not $k$-continuous.
\end{example}

%---------------------------------------------------------------------------------------%
\section{Fisher's quasi-contraction in suprametric  spaces}

In this section we will prove some fixed point results in the setting suprametric  spaces which are generalizations of the results regarding  \'Ciri\'c and Sehgal Fisher's  quasi-contraction.

 In 1979, Fisher in \cite{Fisher} extended the
definition of \'Ciri\'c quasi-contraction as follow and he show if $T$ is a continuous mapping from complete metric space $(X,d)$ into itself satisfying the inequality

$d( T^px,T^qy) \leq \lambda \cdot \max\Big\{d(T^r x,T^s y) ,d(T^r
x,T^{r'}x), d(T^s y,T^{s'}y) :  0 \leq r, r' \leq p, 0\leq s, s' \leq q
\Big\}$\\
for some fixed positive integers $p$ and $q$ and constant $\lambda\in (0,1)$, then $T$ has a unique fixed point.
Also, it is shown that the condition that $T$ be continuous is unnecessary if $q$ (or $p$) =1.

Now, we study Fisher’s results  in the setting of a suprametric space.

The following theorem extends this result in the framework of suprametric spaces.

\begin{theorem}
Let $(X,d)$ be a complete suprametric space and  $T:(X,d)\longrightarrow
(X,d)$ be a self mapping such that
 for some fixed positive integers $n$ and $m$ and some
$\lambda\in [0,1)$ \begin{equation}\label{1f}  d(T^nx,T^my)\leq \lambda\max\{d(T^ix,T^jy),d(T^ix,T^{i'}x),d(T^jy,T^{j'}y): \, 0\leq i,i'\leq n,\,\, 0\leq j,j'\leq m\},\end{equation} 
for every $x,y \in X$.
  Then in each of the following conditions, $T$ has a unique fixed point $u\in X$ such that for each $x_0\in X$, $T^nx_0\to u$.
\begin{itemize}
\item[(i)] $T$ is $k$-continuous, for some $k\in \mathbb{N}$;
\item[(ii)] $\mathfrak{D}(x)= d(x,Tx)$ is $T$-orbitally lower semi-continuous;
\item[(iii)]  $T$ obeys the condition $(C;k)$.
\end{itemize}

\end{theorem}

\begin{proof}

We can assume that $n\ge m$, $\lambda\in (\frac12,1)$ and $1-\rho \frac{\lambda}{1-\lambda} \max\{d(T^ix,T^mx):\, 0\leq i\leq n\}\neq 0$, by changing and increasing it, if necessary.
Also, Let $x\in X$ and $\{T^qx\}_{q=1}^\infty$ is unbounded. Then there is a $N\in \Bbb N$ such that $N$ is smallest one such that

\begin{eqnarray} \label{2}
d(T^Nx,T^mx)& >&\max\Big \{\frac{\lambda}{1-\lambda}\max\{d(T^rx,T^mx):\, 0\leq r\leq n\}\\\nonumber
  &,& \, \frac{\frac{\lambda}{1-\lambda}\max\{d(T^rx,T^mx):\, 0\leq r\leq n\}}{1-\rho\frac{\lambda}{1-\lambda}\max\{d(T^rx,T^mx):\, 0\leq r\leq n\}}\Big\}\\\nonumber
&\geq & \max\{d(T^rx,T^mx):\, 0\leq r\leq n\}.
\end{eqnarray}
Then $ N>n\geq m$
and since $\frac{\lambda}{1-\lambda}\geq 1$
we have
\begin{equation}\label{3}
d(T^Nx,T^mx)\geq \max\{d(T^rx,T^mx):\, 0\leq r<N\}.
\end{equation}
Now since $d$ is suprametric 
$$
d(T^rx,T^ix)\leq d(T^rx,T^mx)+d(T^mx,T^ix)+\rho d(T^rx,T^mx)d(T^mx,T^ix)
$$
Thus
\begin{equation}\label{4}
d(T^rx,T^ix)- d(T^ix,T^mx)\leq d(T^rx,T^mx)(1+\rho d(T^ix,T^mx)),
\end{equation}
where $0\leq i<N$ . Therefore by $\ref{2}$ we have
$$
(1-\lambda)d(T^Nx,T^mx)(1-\rho\frac{\lambda}{1-\lambda}\max\{d(T^rx,T^mx):\, 0\leq r\leq n\})>\lambda \max \{d(T^rx,T^mx):\, 0\leq r\leq n\}
$$
Thus using $\ref{3}$ and $\ref{4}$, we have

\begin{align*}
  (1-\lambda)d(T^Nx,T^mx)&>(1+\rho d(T^Nx,T^mx)) {\lambda}\max\{d(T^rx,T^mx):\, 0\leq r\leq n\}\\
  &\geq(1+\rho d(T^Nx,T^mx)) {\lambda}\max \{\frac{d(T^rx,T^ix)-d(T^ix,T^mx)}{1+\rho d(T^mx,T^ix)}:\, 0\leq r\leq n,\,\, 0\leq i<N\}\\
&>(1+\rho d(T^Nx,T^mx)) {\lambda}\max \{\frac{d(T^rx,T^ix)-d(T^Nx,T^mx)}{1+\rho d(T^mx,T^Nx)}:\, 0\leq r\leq n,\,\, 0\leq i<N\}.
\end{align*}
Thus,
\begin{equation}\label{5}
(1-\lambda)d(T^Nx,T^mx)+{\lambda}d(T^Nx,T^mx)=d(T^Nx,T^mx)>\lambda \max \{d(T^rx,T^ix):\, 0\leq r\leq n,\,\, 0\leq i< N\}.
\end{equation}
Now, we will show that 
\begin{equation}\label{6}
d(T^Nx,T^mx)>\lambda \max \{d(T^rx,T^ix):\,  0\leq i, r< N\}.
\end{equation}
Since otherwise $d(T^Nx,T^mx)\leq\lambda \max \{d(T^rx,T^ix):\, 0\leq i,r<  N\}$ and from $\ref{5}$, we get 
$d(T^Nx,T^mx)\leq\lambda \max \{d(T^rx,T^ix):\, n\leq i,r< N\}$. 
So for each $p\in \mathbb{N}$ and applying $\ref{1f}$ and $\ref{5}$, 
$p$ times, we have 
$d(T^Nx,T^mx)\leq\lambda^p \max \{d(T^rx,T^ix):\, n\leq i,r\leq N\}\to 0$, where $p\to \infty$. Therefore $d(T^Nx,T^mx)=0$, which is a contradiction with $\ref{3}$. So $\ref{6}$ is true.
On the other hand by $\ref{1f}$ we have
\begin{align*}
 d(T^Nx,T^mx)&\leq {\lambda}\max\{d(T^ix,T^jx),d(T^ix,T^{i'}x),d(T^jx,T^{j'}x):\, N-n\leq i,i'\leq N,\, 0\leq j,j'\leq m\}\\\nonumber
  &\leq {\lambda}\max \{d(T^ix,T^jx):\, 0\leq i,j\leq N\};\nonumber
\end{align*}
  a contradiction with $\ref{6}$.
  Hence $\{T^qx\}_{q=1} ^\infty$ must be bounded. So there exists $y\in X$ such that $d(T^qx,y)\leq M$ for some $M>0$ and each $q\in\Bbb N$. Thus
  \begin{align*}
 d(T^qx,T^px)&\leq d(T^qx,y)+d(T^px,y)+\rho d(T^qx,y)d(T^px,y)\\
  &\leq M+M+\rho M^2\\
  &= M_0;
\end{align*}
for each $p,q\in\Bbb N$. Now for each $\varepsilon>0$ and $p,q>\max\{n,m\}N$ such that $\lambda^NM_0<\varepsilon$, we have $$d(T^px,T^qx)\leq \lambda^NM_0<\varepsilon,$$ by applying $\ref{1f}$ $N$ times.
Therefore $\{T^nx\}$ is Cauchy sequence and since $X$ is complete, it is convergent to an element $u\in X$.

Now, similar to Theorem $\ref{t20}$, in each of cases $(i)-(iii)$, we can show that $u$ is a fixed point of $T$.

If $v\in X$ is another fixed point of $T$, then 
  \begin{align*}
 d(u,v)&=d(T^nu,T^mv)\\
 &\leq {\lambda}\max\{d(T^iu,T^ju),d(T^iu,T^{i'}u),d(T^jv,T^{j'}v):\, 0\leq i,i'\leq n,\, 0\leq j,j'\leq m\}\\\nonumber
  &= {\lambda}d(u,v).\nonumber
\end{align*}
Therefore $d(u,v)=0$ and so $u=v$.
\end{proof}

%%%%%%%%%%%%%%%%%%%%%%%%%%%%%%%%%%%%%%%%%%%%%%%%%%%%%

\section{Application}
In this section, we present an application of our main results to the existence and uniqueness of solutions for a Fredholm integral equation.

\begin{theorem}
Let $K:[a,b]\times[a,b]\to\mathbb{R}$ be a continuous kernel such that $|K(x,t)|\leq M$ for all $x,t\in[a,b]$. Consider the integral equation
\begin{equation}\label{eq:integral}
f(x)=\int_a^b K(x,t)f(t)\,dt, \quad x\in[a,b].
\end{equation}
If $L = M(b-a) < 1$, then \eqref{eq:integral} has a unique solution in $C[a,b]$.
\end{theorem}

\begin{proof}
Let $X=C[a,b]$ be equipped with the supremum norm $\|f\|_\infty = \sup_{x\in[a,b]}|f(x)|$. Fix $\lambda>0$ and define $d:X\times X\to[0,\infty)$ by
\[
d(f,g) = \|f-g\|_\infty \bigl(\|f-g\|_\infty + \lambda\bigr).
\]
By Example \ref{Berzig1}, $(X,d)$ is a complete suprametric space with constant $\rho = 2/\lambda$.

Define the operator $T:X\to X$ by $(Tf)(x) = \int_a^b K(x,t)f(t)\,dt$. Since $K$ is continuous, $T$ maps $C[a,b]$ into itself. For any $f,g\in X$, we estimate
\[
\|Tf-Tg\|_\infty \leq \sup_{x\in[a,b]} \int_a^b |K(x,t)|\,|f(t)-g(t)|\,dt \leq M(b-a)\|f-g\|_\infty = L\|f-g\|_\infty.
\]
Iterating this bound yields $\|T^2f-T^2g\|_\infty \leq L^2\|f-g\|_\infty$. We now verify that $T$ satisfies the convex contraction condition of order $m=2$ from Theorem \ref{t20}. Set $a_0=L^2$ and $a_1=0$. Since $L<1$, we have $a_0+a_1 = L^2 < 1$. Then,
\begin{align*}
d(T^2f,T^2g) &= \|T^2f-T^2g\|_\infty \bigl(\|T^2f-T^2g\|_\infty + \lambda\bigr) \\
&\leq L^2\|f-g\|_\infty \bigl(L^2\|f-g\|_\infty + \lambda\bigr) \\
&\leq L^2\|f-g\|_\infty \bigl(\|f-g\|_\infty + \lambda\bigr) \\
&= a_0 d(f,g) + a_1 d(Tf,Tg).
\end{align*}
Thus, $T$ is a convex contraction of order $2$ on the complete suprametric space $(X,d)$. Furthermore, the boundedness of $K$ implies that $T$ is Lipschitz continuous, and hence $T$ is $k$-continuous for any $k\in\mathbb{N}$ (satisfying condition (i) of Theorem \ref{t20}). By invoking Theorem \ref{t20}, $T$ possesses a unique fixed point $f^*\in X$, which corresponds to the unique solution of the integral equation \eqref{eq:integral}.
\end{proof}

%%%%%%%%%%%%%%%%%%%%%%%%%%%%%%%%%%%%%%%%%%%%%%%%%%%%%

%---------------------------------------------------------------------------------------%
\bigskip
%\section{Conclusion}
%In the present  paper, we have proved some new fixed point theorems using weak convex contraction
%mapping of order m $(m \geq 2)$ on a complete suprametric space but do not force
%the mapping to be continuous at the fixed point.

%\fbox{\scriptsize{\textbf{Please cite your relevant papers but at most total 10 papers/books.}}}
\hspace{1in}

%%%%%%%%%%%%%%%%%%%%%%%%%%%%%%%%%%%%%%%%%%%
% References
%%%%%%%%%%%%%%%%%%%%%%%%%%%%%%%%%%%%%%%%%%%

%%%%%%%%%%%%%%%%%%%%%%%%%%%%%%%%%%%%%%%%%%%%%%%%%%%%%%%%%%%%%%%
%%%%%%%%%%%%%%%%%%%%%%%%%%%%%%%%%%%%%%%%%%%%%%%%%%%%%%%%%%%%%%%

 \end{document}